\documentclass[11pt,a4paper]{article}

\usepackage{amsmath,amsfonts,amssymb,amsthm}
\usepackage{bbm,mathrsfs, bm, mathtools}

\title{Random zonotopes and valuations}

\author{Rolf Schneider}
\date{}
\newcommand{\R}{{\mathbb R}}

\newcommand{\K}{{\mathcal K}}

\newcommand{\bP}{{\mathbb P}}

\newcommand{\N}{{\mathbb N}}

\newcommand{\D}{{\rm d}}

\newcommand{\bE}{{\mathbb E}\,}

  \renewcommand{\dim}{{\rm dim}\,}

\newtheorem{theorem}{Theorem}
\newtheorem{lemma}{Lemma}

\begin{document}
\maketitle

\begin{abstract}
We define a random zonotope in $\R^d$, by adding finitely many random segments, which are independently and identically distributed. For this random polytope, we determine, under a mild assumption on the distribution, the expectations of the intrinsic volumes, more generally, the expectations of suitable valuations. We also prove a central limit theorem for a valuation evaluated at these random zonotopes.
\\[1mm]
\end{abstract}

\section{Introduction}\label{sec1}

In the stochastic geometry of $\R^d$, many constructions of random polytopes start with a finite number of independent, identically distributed random points. A common procedure is to take their convex hull. A first task will then be to determine the expectations of some prominent geometric functionals of the random polytope, for example, the numbers of $k$-dimensional faces or the intrinsic volumes. Often only asymptotic considerations, with the number of points or the dimension $d$ tending to infinity, lead to explicit results. (Surveys on random polytopes are, in chronological order, \cite{Buc85}, \cite{Sch88}, \cite[Sect. 8.2]{SW08}, \cite{Sch08}, \cite{Rei10}, \cite{Hug13}, \cite{Sch18}, \cite[Sect. 1.4]{Cal19}). In the present note, we do not study the convex hull, but the Minkowski sum $Z_p$ of $p$ segments with endpoints at the origin $o$ and at the random points. Thus, we consider random zonotopes. (Surveys on zonotopes and on zonoids, their limits in the Hausdorff metric, are \cite{SW83} and \cite{GW93}.) In this case, the numbers of $k$-dimensional faces are not a problem, since under weak assumptions on the distribution of the random points they are almost surely constant. We refer to Donoho and Tanner \cite[(1.6)]{DT10} and the references given there, although the formulation there is slightly different (but equivalent). 

In contrast, the expectations of the intrinsic volumes of the random zonotope $Z_p$ will strongly depend on the given distribution. The intrinsic volumes are of particular interest, since they include the most common size measurements  of convex bodies, namely volume, surface area, and mean width (up to normalizing constants). We show in this note that the expectations of the intrinsic volumes of $Z_p$ can be expressed via a deterministic zonoid constructed from the distribution. In fact, the expectations of more general functionals than the intrinsic volumes can be made explicit in this way.

To be more precise, let $\R^d$ be the $d$-dimensional Euclidean vector space, with scalar product $\langle\cdot\,,\cdot\rangle$ and induced norm $\|\cdot\|$. By $\K^d$ we denote the space of convex bodies (nonempty, compact, convex subsets) of $\R^d$, equipped with the Hausdorff metric. We recall, e.g. from \cite{Sch14}, that the intrinsic volumes $V_0(K),\dots,V_d(K)$ of a convex body $K\in\K^d$ can be defined via the Steiner formula
$$ V_d(K+\rho B^d) = \sum_{j=0}^d \rho^{d-j}\kappa_{d-j}V_j(K).$$
Here $V_d$ denotes the volume, $\rho\ge 0$ can be any real number, $B^d$ is the unit ball in $\R^d$, and $\kappa_d$ is its volume. 

The intrinsic volume $V_j$ is an element of ${\bf Val}$, the real vector space of translation invariant, continuous real valuations on the space $\K^d$. Recall that $\varphi:\K^d\to\R$ is a valuation if $\varphi(K\cup L)+\varphi(K\cap L)=\varphi(K)+\varphi(L)$ whenever $K,L,K\cup L\in\K^d$. (Surveys on valuations are \cite{McMS83}, \cite{McM93}, \cite[Chap. 6]{Sch14}.) By Hadwiger's theorem, the intrinsic volumes form a basis of the subspace of rigid motion invariant valuations. The space ${\bf Val}$ is much bigger, it has infinite dimension. By a theorem of P. McMullen, ${\bf Val}=\bigoplus_{j=0}^d {\bf Val}_j$, where ${\bf Val}_j$ is the subspace of valuations that are homogeneous of degree $j$. For example, each function
$$ K\mapsto V(\underbrace{K,\dots,K}_j,L_{j+1},\dots,L_d),$$
where $V$ denotes the mixed volume in $\R^d$ and $L_{j+1},\dots,L_d$ are fixed convex bodies, belongs to ${\bf Val}_j$. The $j$th intrinsic volume is obtained, up to a a normalizing factor, if $L_{j+1},\dots,L_d$ are unit balls. The elements of ${\bf Val}_j$ appear in the subsequent Theorem \ref{T1} (only for $j\ge 1$, since ${\bf Val}_0$ contains only constants).

For $x\in\R^d$, we denote by $\overline x$ the closed segment with endpoints $x$ and $o$.

Let $X$ be a random vector in $\R^d$. Of the distribution of $X$, we assume that $\bE\|X\|<\infty$. For $p\in\N$, let $X_1,\dots,X_p$ be i.i.d. copies of $X$, and define
$$ Z_p:= \frac{1}{p}\left(\overline X_1+\dots+\overline X_p\right),$$
where the sum is Minkowski (or vector) addition.

The random set $\overline X$ is integrably bounded and hence has a selection expectation $\bE \overline X$ (we refer to Molchanov \cite{Mol05}, in particular Theorem 1.22 of Chapter 2). We denote this selection expectation by ${\sf Z}_X$. We recall that the support function of a convex body $K\in\K^d$ is defined by $h(K,u):= \max\{\langle x,u\rangle:x\in K\}$ for $u\in\R^d$. It is known that the support function of the selection expectation ${\sf Z}_X$ is given by
$$ h({\sf Z}_X,u)= \bE h(\overline X,u) = \int_{\R^d} h(\overline x,u)\,\bP_X(\D x)\quad\mbox{for }u\in\R^d,$$
where $\bP_X$ denotes the distribution of $X$. Thus, ${\sf Z}_X$ can be approximated by finite sums of segments and hence is a zonoid. We remark that ${\sf Z}_X$ occurs also as the `zonoid of a measure'; see Mosler \cite{Mos02}, in particular Theorem 2.8.

\begin{theorem}\label{T1}
Let $j\in\{1,\dots,d\}$, and let $p\ge j$ be an integer. If $\varphi \in{\bf Val}_j$, then
\begin{equation}\label{Z6}
\bE \varphi(Z_p) =\frac{p!}{p^j(p-j)!}\,\varphi({\sf Z}_X).
\end{equation}
\end{theorem}

We mention that the role of the zonoid ${\sf Z}_X$ can be seen in a vague analogy to that of Matheron's \cite{Mat75} `Steiner compact' (called the `associated zonoid' in \cite{SW08}) in the theory of Poisson hyperplane tessellations. The common aspect is that the distribution of a random object is governed by a measure, and this measure defines a deterministic zonoid. This zonoid can then be used to express certain expectations defined by the random object.

We prove Theorem \ref{T1} in Section \ref{sec3}, after explaining in the next section the ideas of Vitale on which it is based. Section 4 contains a central limit theorem for $\varphi(Z_n)$, as $n$ tends to infinity.

\section{A result of Richard A. Vitale}\label{sec2}

The following result was proved by Vitale \cite{Vit91}.

\begin{theorem}\label{T2} 
If $X_1,\dots,X_d$ are i.i.d. copies of $X$, then
\begin{equation}\label{Z1}
\bE V_d(\overline X_1+\dots+\overline X_d)=d!V_d({\sf Z}_X).
\end{equation}
\end{theorem}

Vitale's formulation is slightly different, since he considers not the volume of the sum of $d$ random segments (and thus of a random parallelepiped), but the absolute determinant of a $d\times d$ matrix whose columns are i.i.d. random. The formulations are clearly equivalent.

Vitale's result appears also as Theorem 2.10 in \cite{Mos02}.

Vitale's proof is an elegant combination of two strong laws of large numbers. We recall his argument briefly, since we are going to extend it. Let $X_1,X_2,\dots$ be an infinite sequence of i.i.d. copies of $X$. For each $n\in\N$, consider the random zonotope $Z_n:= \frac{1}{n}\left(\overline X_1+\dots+\overline X_n\right)$. Since $\bE \|X\|<\infty$, the strong law of Artstein and Vitale \cite{AV75} can be applied, which yields that, with probability one,
\begin{equation}\label{Z3}
Z_n \to \bE \overline X = {\sf Z}_X \quad \mbox{as $n\to\infty$,}
\end{equation}
where the convergence is in the Hausdorff metric. Since $V_d$ is continuous on convex bodies, we get
\begin{equation}\label{Z4}
\lim_{n\to\infty} V_d(Z_n)= V_d({\sf Z}_X) \quad\mbox{a.s.}
\end{equation}
We have (see, e.g., \cite[p. 304]{Sch14})
$$ V_d(Z_n)=\frac{1}{n^d} \sum_{1\le i_1<\dots< i_d\le n} V_d(\overline X_{i_1}+\dots+\overline X_{i_d}).$$
Here, $V_d(\overline X_{i_1}+\dots+\overline X_{i_d})\le \|X_{i_1}\|\cdots\|X_{i_d}\|$, hence the random variable $V_d(Z_n)$ has finite expectation. From Hoeffding's strong law of large numbers for $U$-statistics (see, e.g., Serfling \cite[Chap. 5]{Ser80}) it follows that, with probability one,
$$\binom{n}{d}^{-1}\sum_{1\le i_1<\dots<i_d\le n} V_d(\overline X_{i_1}+\dots+ \overline X_{i_d}) \to \bE V_d(\overline X_1+\dots+\overline X_d)\quad\mbox{as }n\to\infty,$$ 
hence
\begin{equation}\label{Z5} 
\lim_{n\to\infty} V_d(Z_n)=\frac{1}{d!}\bE V_d(\overline X_1+\dots+\overline X_d) \quad\mbox{a.s.}
\end{equation}
From (\ref{Z4}) and (\ref{Z5}), Vitale's result (\ref{Z1}) follows. We remark that similar arguments were used in \cite{Vit95}.

\section{An extension}\label{sec3}

Since we have given Vitale's result a geometric formulation, we can extend it in the way described in Theorem \ref{T1}. The proof depends on a polynomiality property of the elements of ${\bf Val}_j$, which goes back to P. McMullen.

Let $j\in\{1,\dots,d\}$ and $\varphi\in{\bf Val}_j$. Then there exists a continuous symmetric mapping $\phi:(\K^d)^j\to \R$ which is translation invariant and Minkowski additive in each variable, such that
\begin{equation}\label{V1} 
\varphi(\lambda_1K_1+\dots+\lambda_n K_n)=\sum_{r_1,\dots,r_n=0}^j \binom{j}{r_1\dots r_n}\lambda_1^{r_1}\cdots\lambda_n^{r_n}\phi(K_1[r_1],\dots,K_n[r_n])
\end{equation}
for all $K_1,\dots,K_n\in\K^d$, all $\lambda_1,\dots,\lambda_n\ge 0$ and all $n\in\N$. Here the multinomial coefficient is, by definition, equal to $0$ if $r_1+\dots+r_n\not=j$. The bracket $[r_i]$ indicates that the preceding argument is repeated $r_i$ times. We refer to \cite[Thm. 6.3.6]{Sch14}, and to the subsequent note for references and some history. 

Suppose that all $K_i$ in (\ref{V1}) are segments. Consider a summand on the right-hand side where some $r_i$ is greater than $1$, say $r_1\ge 2$. This summand contains an expression $\phi(\overline x_1,\overline x_1,\overline y_3,\dots,\overline y_j)$ with $x_1,y_3,\dots,y_j\in\R^d$. For arbitrary $\lambda_1,\dots,\lambda_j\ge 0$ we have
$$ \dim(\lambda_1\overline x_1+\lambda_2\overline x_1+\lambda_3\overline y_3+\dots+\lambda_j\overline y_j)<j,$$
hence $\varphi(\lambda_1\overline x_1+\lambda_2\overline x_1+\lambda_3\overline y_3+\dots+\lambda_j\overline y_j)=0$, by \cite[Cor. 6.3.2]{Sch14} and the continuity of $\varphi$. Therefore, it follows from (\ref{V1}) that $\phi(\overline x_1,\overline x_1,\overline y_3,\dots,\overline y_j)=0$. This means that for arbitrary segments $\overline x_1,\dots,\overline x_n$ we have
$$ \phi(\overline x_1[r_1],\dots,\overline x_n[r_n]) =0 $$
whenever $r_1+\dots+ r_n=j$ and at least on $r_i$ is greater than one. Therefore, if (\ref{V1}) is applied to segments, then on the right-hand side only the summands with $r_i\in\{0,1\}$ can be different from zero. It follows that for segments $\overline x_1,\dots,\overline x_n$, relation (\ref{V1}) can be simplified to
\begin{equation}\label{V3}
\varphi(\overline x_1+\dots+\overline x_n) =j!\sum_{1\le i_1<\dots< i_j\le n} \phi(\overline x_{i_1},\dots,\overline x_{i_j}).
\end{equation}

Let $p$ be an integer with $j\le p\le n$ and consider
$$ A_p:= \sum_{1\le i_1<\dots< i_p\le n} \varphi(\overline x_{i_1}+\dots+\overline x_{i_p}).$$
(Only the case $p\ge j$ is considered, since $\varphi(\overline x_1+\dots+\overline x_p)=0$ for $p<j$.) By (\ref{V3}) (with $\overline x_1+\dots+\overline x_n$ replaced by $\overline x_{i_1}+\dots+ \overline x_{i_p}$) we get
$$ A_p =j!\sum_{1\le i_1<\dots< i_p\le n}  \,\sum_{\genfrac{}{}{0pt}{}{k_1<\dots<k_j}{k_1,\dots,k_j\in\{i_1,\dots,i_p\}}} \phi(\overline x_{k_1},\dots,\overline x_{k_j}).$$
A given ordered $j$-tuple $k_{1}<\dots< k_{j}$ appears in the summation as often as one can choose a set of $p-j$ distinct indices from $\{1,\dots, n\}\setminus \{k_1,\dots,k_j\}$, that is, $\binom{n-j}{p-j}$ times. In view of (\ref{V3}) it follows that $A_p=\binom{n-j}{p-j}\varphi(\overline x_1+\dots+\overline x_n)$ and thus
\begin{equation}\label{Z11a}
\varphi(\overline x_1+\dots+\overline x_n) =\binom{n-j}{p-j}^{-1} \sum_{1\le i_1<\dots<i_p\le n} \varphi(\overline x_{i_1}+\dots+\overline x_{i_p}).
\end{equation}

Now, as in Section \ref{sec2}, let $X_1,X_2,\dots$ be i.i.d. copies of $X$ and define 
$$Z_n:=\frac{1}{n}(\overline X_1+\dots+\overline X_n).$$ 
Then (\ref{Z11a}) implies that
\begin{equation}\label{Z12}
\varphi(Z_n) = \frac{1}{n^j}\binom{n-j}{p-j}^{-1}\binom{n}{p} U_n^{(p)}(h),
\end{equation}
where 
$$ U_n^{(p)}(h):= \binom{n}{p}^{-1} \sum_{1\le i_1<\dots < i_p\le n} h(X_{i_1},\dots,X_{i_p})$$
is a $U$-statistic of order $p$ with kernel function $h:(\R^d)^p\to\R$ defined by
$$ h(x_1,\dots,x_p):= \varphi(\overline x_1+\dots+\overline x_p)$$
(and applied to the i.i.d. sequence $X_1,X_2,\dots$).

Now we remark that the function $\phi$ is in each variable Minkowski additive and continuous, hence positively homogeneous of degree $1$. Therefore,
$$ \phi(\overline X_1,\dots,\overline X_j) = \|X_1\|\cdots\|X_j\|\phi(s_1,\dots,s_j)$$
with random segments $s_1,\dots,s_j$ having endpoints at $o$ and at unit vectors. Since $\phi$ is continuous, there is a constant $c$, depending only on $\phi$, such that $|\phi(s_1,\dots,s_j)|\le c$. Since $\bE\|X\|<\infty$ and $X_1,\dots,X_j$ are stochastically independent, we deduce that $\bE |\phi(\overline X_1,\dots,\overline X_j)|<\infty$. From (\ref{V3}) we now conclude that $\bE|h(X_1,\dots,X_p)|=\bE |\varphi(\overline X_1+\dots+\overline X_p)|<\infty$.

Therefore, the strong law for $U$-statistics (see, e.g., Serfling \cite[p. 190]{Ser80} or Lee \cite[p. 122]{Lee90}) can be applied. It yields that, with probability one,
$$ \lim_{n\to\infty}  U_n^{(p)}(h)= \bE h(X_1,\dots,X_p) = \bE \varphi(\overline X_{1}+\dots+ \overline X_{p})= p^j\, \bE \varphi(Z_p).$$
Therefore (\ref{Z12}), together with
$$ \lim_{n\to\infty} \frac{1}{n^j}\binom{n-j}{p-j}^{-1}\binom{n}{p} = \frac{(p-j)!}{p!},$$
gives
$$ \lim_{n\to\infty} \varphi(Z_n)=\frac{(p-j)!}{p!} p^j\,\bE\varphi(Z_p) \quad\mbox{a.s.}$$ 

By (\ref{Z3}) and the continuity of $\varphi$, we have
\begin{equation}\label{Z12a} 
\lim_{n\to\infty} \varphi(Z_n) =\varphi({\sf Z}_X)\quad\mbox{a.s.}
\end{equation}
Both results together yield the value for the expectation $\bE\varphi(Z_p)$ as stated in Theorem \ref{T1}.

\section{A central limit theorem}\label{sec4}

Besides the strong law of large numbers expressed by (\ref{Z12a}), we may also state a central limit theorem for $\varphi(Z_n)$. For this, we need to compute a variance in our present case. We consider the kernel function $h$ for $p=j$, that is, $h(x_1,\dots,x_j)=\varphi(\overline x_1+\dots+\overline x_j)$. We assume now that $\bE\|X\|^2<\infty$. Then we can conclude as in Section \ref{sec2}, using (\ref{V3}), that
$$ \bE h^2(X_1,\dots,X_j)=\bE(j!\phi(\overline X_1,\dots,\overline X_j))^2 =(j!)^2\bE\left(\|X_1\|^2\cdots\|X_j\|^2\phi(s_1,\dots,s_j)^2\right)$$
with segments $s_1,\dots,s_j$ of unit length, and hence that $\bE h^2(X_1,\dots,X_j)<\infty$.

Adopting the notation of \cite{Ser80}, we write 
$$\theta:= \bE h(X_1,\dots,X_j)= \bE\varphi(\overline X_1+\dots+\overline X_j)=j^j\,\bE \varphi(Z_j)=j!\varphi({\sf Z}_X)$$
and, for $x_1\in\R^d$,
$$ h_1(x_1) := \bE h(x_1,X_2,\dots,X_j)=\bE \varphi(\overline x_1+\overline X_2+\dots+\overline X_j),$$
further $\widetilde h_1:=h_1-\theta$ and 
$$ \zeta_1:= \bE \widetilde h_1^2(X_1).$$

Now we define a random zonotope by
$$ Z_n(x_1) := \overline x_1 + \frac{1}{n}(\overline X_2+\dots+\overline X_n) \stackrel{d}{=} \overline x_1+\frac{n-1}{n}Z_{n-1}$$
for $x_1\in\R^d$. Then by (\ref{Z3}), with probability one,
$$ Z_n(x_1) \to \overline x_1+\bE \overline X =\overline x_1+{\sf Z}_X\quad\mbox{as }n\to\infty$$
and hence
\begin{equation}\label{Z20}
\lim_{n\to\infty} \varphi(Z_n(x_1)) = \varphi(\overline x_1+{\sf Z}_X) \quad\mbox{a.s.}
\end{equation}

We write (\ref{Z11a}) for $p=j$ in the form
\begin{eqnarray*} 
&&\varphi(\overline x_1+\dots+\overline x_n)\\
&&=\sum_{2\le i_2<\dots<i_j\le n} \varphi(\overline x_1+\overline x_{i_2}+\dots+\overline x_{i_j}) + \sum_{2\le i_1<\dots<i_j\le n} \varphi(\overline x_{i_1}+\dots+\overline x_{i_j}).
\end{eqnarray*}
This gives 
\begin{eqnarray*} 
&&\varphi(Z_n(x_1))\\
&&=\sum_{2\le i_2<\dots<i_j\le n} \varphi\left(\overline x_1+\frac{1}{n}\overline X_{i_2}+\dots+\frac{1}{n}\overline X_{i_j}\right) + \sum_{2\le i_1<\dots<i_j\le n} \varphi\left(\frac{1}{n}\overline X_{i_1}+\dots+\frac{1}{n}\overline X_{i_j}\right)\\
&&= \frac{1}{n^{j-1}}\sum_{2\le i_2<\dots<i_j\le n} \varphi\left(\overline x_1+\overline X_{i_2}+\dots+\overline X_{i_j}\right) + \frac{1}{n^j}\sum_{2\le i_1<\dots<i_j\le n} \varphi\left(\overline X_{i_1}+\dots+\overline X_{i_j}\right).
\end{eqnarray*}
Here we have used that $\varphi$ is homogeneous of degree $j$ and that (\ref{V3}), together with the Minkowski-linearity of $\phi$ in each argument, implies that for segments $S_1,\dots,S_j$ we have $\varphi(nS_1+S_2+\dots+S_j)=n\varphi(S_1+\dots+S_j)$.

The result can be written as
$$ \varphi(Z_n(x_1)) = \frac{1}{n^{j-1}} \binom{n-1}{j-1} U_{n-1}^{(j-1)}(g_{x_1}) + \frac{1}{n^{j}} \binom{n-1}{j} U_{n-1}^{(j)}(h)$$
with $g_{x_1}(x_2,\dots,x_j)= \varphi(\overline x_1+\overline x_2+\dots+\overline x_j)$. Therefore, the strong law for $U$-statistics gives that, with probability one,
$$ \lim_{n\to\infty} \varphi(Z_n(x_1))=\frac{1}{(j-1)!}\,\bE\varphi(\overline x_1+\overline X_2+\dots+\overline X_j) +
\frac{1}{j!}\,\bE\varphi(\overline X_1+\dots+\overline X_j)].$$
Together with (\ref{Z20}) this yields
$$ h_1(x_1)=(j-1)!\left[\varphi(\overline x_1+{\sf Z}_X)-\varphi({\sf Z}_X)\right].$$

The central limit theorem requires that $\zeta_1>0$. To achieve this, we need assumptions on the distribution of the random vector $X$ and on the valuation $\varphi$.

\begin{lemma}\label{L4.1}
Let $j\in\{1,\dots,d\}$ and $\varphi\in{\bf Val}_j$. Suppose that the support of the distribution $\bP_X$ of $X$ contains the origin and is not contained in a $(j-1)$-dimensional linear subspace. Suppose further that $\varphi(K)\not=0$ for convex bodies $K$ with $\dim K\ge j$ (as it is satisfied for the $j$th intrinsic volume). Then $\zeta_1>0$.
\end{lemma}

\begin{proof}
Since the distribution $\bP_X$ of $X$ is not concentrated on a $(j-1)$-dimensional linear subspace, the zonoid ${\sf Z}_X$ has dimension at least $j$. Assume that $\zeta_1=0$. Since $h_1$ is continuous, we then have $(h_1(x_1)-\theta)^2=0$ for all $x_1$ in the support of $\bP_X$ and hence 
$$ \varphi(\overline x_1+{\sf Z}_X)= (j+1) \varphi({\sf Z}_X)$$
for these $x_1$. Since $o$ is in the support of $\bP_X$, this yields $\varphi({\sf Z}_X)=0$, a contradiction. 
\end{proof}

We can now formulate a central limit theorem.

\begin{theorem}\label{T4.1} 
Suppose that $\bE \|X\|^2<\infty$. Under the assumptions of Lemma \ref{L4.1},
$$ \sqrt{n}\left(\varphi(Z_n) -\varphi({\sf Z}_X)\right) \stackrel{d}{\to} {\mathcal N}(0, (j!j)^2\zeta_1).$$
\end{theorem}

\begin{proof}
We have seen above that $\bE \|X\|^2<\infty$ implies $\bE h^2(X_1,\dots,X_j)<\infty$. The assumptions of Lemma \ref{L4.1} imply that $\zeta_1>0$. Therefore, Hoeffding's central limit theorem for $U$-statistics (see \cite[p. 192]{Ser80}, \cite[p. 76]{Lee90}, \cite[p. 128]{KB94}) says that
$$ \sqrt{n}\left(U_n^{(j)}(h) -\theta\right)  \stackrel{d}{\to} {\mathcal N}(0, j^2\zeta_1).$$ 
Since 
$$ \varphi(Z_n)= \frac{1}{n^j}\binom{n}{j} U_n^{(j)}(h),\qquad \theta = j!\varphi({\sf Z}_X),$$
Slutsky's theorem (e.g., \cite[Sect. 1.5.4]{Ser80}) gives the assertion.
\end{proof}

Finally, we want to show how in a simple case the expectation and variance appearing in Theorem \ref{T4.1} can be computed explicitly. We assume that $X$ has a standard normal distribution and $\varphi=V_j$, the $j$th intrinsic volume, for some $j\in\{1,\dots,d\}$.

Since $h(\overline x,u)+ h(-\overline x, u)= |\langle x,u\rangle|$ and the distribution of $X$ is invariant under reflection in the origin, we get
$$ h({\sf Z}_X,u) = \int_{{\mathbb R}^d} h(\overline x,u)\,\bP_X(\D x)= \frac{1}{2\sqrt{2\pi}^d} \int_{{\mathbb R}^d} |\langle x,u\rangle|e^{-\frac{1}{2}\|x\|^2}\,\D x.$$ 
Using polar coordinates and denoting the spherical Lebesgue measure on the unit sphere $S^{d-1}$ by $\sigma_{d-1}$, we obtain
\begin{eqnarray*}
h({\sf Z}_X,u) &=& \frac{1}{2\sqrt{2\pi}^d} \int_{S^{d-1}} \int_0^\infty |\langle rv,u\rangle|e^{-\frac{1}{2}r^2}r^{d-1}\,\D r\,\sigma_{d-1}(\D v)\\
&=& \frac{1}{2\sqrt{2\pi}^d} \int_{S^{d-1}} |\langle v,u\rangle|\,\sigma_{d-1}(\D v)  \int_0^\infty e^{-\frac{1}{2}r^2}r^d\,\D r\\
&=& \frac{1}{2\sqrt{2\pi}^d}\cdot\frac{\kappa_{d-1}}{2}\cdot 2^{\frac{d-1}{2}}\Gamma\left(\frac{d+1}{2}\right).
\end{eqnarray*}
Thus, ${\sf Z}_X$ is a ball with center $o$ and radius
$$ {\sf R}= \frac{1}{4\sqrt{2\pi}}.$$
It follows from the Steiner formula for the volume (e.g., \cite[(4.1)]{Sch14}) that
$$ V_j({\sf Z}_X)=\binom{d}{j}\frac{\kappa_d}{\kappa_{d-j}}{\sf R}^j.$$

To compute the function $h_1$, given by
$$ h_1(x)= (j-1)!\left[V_j(\overline x   +{\sf Z}_X)-V_j({\sf Z}_X)\right]$$
for $x\in\R^d$, we note that ${\sf Z}_X={\sf R}B^d$ and
$$ V_j(K+{\sf R}B^d) =\sum_{r=0}^j \binom{d-r}{j-r} \frac{\kappa_{d-r}}{\kappa_{d-j}}V_r(K){\sf R}^{j-r}$$
for convex bodies $K$, as can be deduced from the Steiner formula for the volume. If $K=\overline x$, we have $V_0(K)=1$, $V_1(K)=\|x\|$ and $V_r(K)=0$ for $r\ge 2$. This gives
$$ h_1(x)= \frac{(d-1)!}{(d-j)!}\frac{\kappa_{d-1}}{\kappa_{d-j}}{\sf R}^{j-1}\|x\|$$
and
$$ \widetilde h_1(x) = h_1(x)-j!V_j({\sf Z}_X) = a(\|x\|-b)$$
with 
$$ a= \frac{(d-1)!}{(d-j)!} \frac{\kappa_{d-1}}{\kappa_{d-j}} {\sf R}^{j-1},\qquad b= \frac{d\kappa_d}{\kappa_{d-1}}{\sf R}.$$
Thus,
$$ \zeta_1= a^2\left[\bE\|X\|^2-2b\,\bE\|X\|+b^2\right]$$
with
$$ \bE \|X\|^2= d,\qquad \bE \|X\|= \frac{d}{\sqrt{2\pi}}\frac{\kappa_d}{\kappa_{d-1}},$$
which altogether yields an explicit expression for $\zeta_1$.

\noindent Author's address:\\[2mm]
Rolf Schneider\\Mathematisches Institut, Albert--Ludwigs-Universit{\"a}t\\D-79104 Freiburg i.~Br., Germany\\E-mail: rolf.schneider@math.uni-freiburg.de

\end{document}